\documentclass{amsart}

\newcommand{\cal}{foo}
\usepackage{john}

\newcommand{\prim}{\oper{prim}}

\usepackage[all,cmtip]{xy}

\title{Ordinary modular forms and companion points on the eigencurve}
\author{John Bergdall}
\date{\today}
\email{bergdall@brandeis.edu}
\urladdr{http://people.brandeis.edu/~bergdall}
\address{John Bergdall\\ Brandeis University \\ MS 050 \\415 South Street\\ Waltham, MA 02454\\USA}
\thanks{The author would like to heartily thank his Ph.D. advisor Jo\"el Bella\"iche, first for his suggestion to think about the problem of companion forms as a question about the eigencurve and second for his energy, enthusiasm and advice while I was preparing early drafts of this article. Thanks are also due to the anonymous referee for many helpful suggestions and corrections.}

\begin{document}

\begin{abstract}
We give a new proof of a result due to Breuil and Emerton which relates the splitting behavior at $p$ of the $p$-adic Galois representation attached to a $p$-ordinary modular form to the existence of an overconvergent $p$-adic companion form for $f$.
\end{abstract}

\maketitle

\section{Introduction}\label{sec:intro}
The goal of this paper is to give a new proof of the theorem due to Breuil and Emerton \cite[Theorem 1.1.3]{BreuilEmerton-Ordinary} that the critical refinement of a $p$-ordinary modular form $f$ has a companion form $g$ if and only if the $p$-adic representation attached to $f$ is split at $p$. The original proof of Breuil and Emerton uses rigid analytic geometry to explicitly construct an overconvergent $p$-adic companion form in the cohomology of modular curves. The proof given here, on the other hand, uses completely different tools. Our technique is to view $f$ in a $p$-adic family on the eigencurve and study how Galois theoretic properties, namely the Hodge-Tate weights, vary infinitesimally. The argument we give may be summarized by saying that $p$-ordinary forms which are split at $p$ are distinguished from other points by the presence of ramification of the eigencurve over the weight space. In the introduction we will give a precise statement of the theorem, an indication of the proof and a comparison of this proof to prior ones.

Fix a prime $p$. Let $f$ be a normalized cuspidal newform of level $\Gamma_1(Np^r)$ with $(N,p) = 1$ and $r \geq 0$ an integer. Let $k \geq 2$ be the weight of $f$ and denote by $\ge$ its nebentypus character. In this notation, $N$ is known as the tame level of $f$. We will write the $q$-expansion of $f$ at infinity as $\sum_{n\geq 1} a_n(f) q^n$. It is well known that there exists a Galois representation
\begin{equation*}
\bar \rho_f : G_{\Q,Np} \goto \GL_2(\bar \Q_p)
\end{equation*}
uniquely characterized by the fact $\bar \rho_f$ is unramified at $\ell \ndvd Np$, and for such $\ell$ the characteristic polynomial of the image of a geometric Frobenius element $\bar \rho_f(\Frob_\ell)$ is $X^2 - a_\ell(f)X + \ell^{k-1}\ge(\ell)$.  The representation $\bar \rho_f$ is  irreducible, de Rham at $p$ and crystalline at $p$ if $r = 0$.

Recall that we say that $f$ is ordinary at $p$ if $a_p(f)$ is a $p$-adic unit. An equivalent criterion is that the representation $\bar \rho_{f,p} := \restrict{\bar \rho_f}{G_{\Q_p}}$ is ordinary in the sense of Greenberg, i.e. up to a change of basis we have that
\begin{equation*}
\bar \rho_{f,p} = \begin{pmatrix} \eta & * \\  & \eta^{-1} \chi_p^{1-k} \end{pmatrix}
\end{equation*} 
where $\eta$ is an unramified character and $\chi_p$ is the $p$-adic cyclotomic character. Since ordinary representations are by definition reducible, it is natural to ask when such representations are split. The criterion of Breuil-Emerton gives necessary and sufficient conditions for $\bar \rho_f$ being split at $p$ in terms of overconvergent $p$-adic compaion forms. In order to explain this, we also have to briefly recall (but see the beginning of \S \ref{sec:proof} for more) the setup of the eigencurve. 

Let $\Gamma := \Gamma_1(N)\intersect \Gamma_0(p)$ and $\cal H$ be the abstract Hecke algebra generated by the operators $(T_\ell)_{\ell \ndvd Np}, U_p$ and $(\langle a \rangle)_{a \ndvd Np}$ for the group $\Gamma$. The tame level $N$ eigencurve $\cal C$ is a rigid analytic curve equipped with a locally finite and flat map $\kappa: \cal C \goto \cal W := \Hom(\Z_p^\x, \G_m)$. A point $x \in \cal C(\bar \Q_p)$ of weight $\kappa(x)$ corresponds to a finite slope system of eigenvalues, denoted $x: \cal H \goto \bar \Q_p$, for the Hecke algebra acting on the space $M_{\kappa(x)}^\dagger(\Gamma)$ of overconvergent $p$-adic modular forms of weight $\kappa(x)$. Given a point $x \in \cal C(\bar \Q_p)$ such that $\log \kappa(x) = k \geq 2$ is an integer, we say that $x$ has a companion point on $\cal C$ if there exists a point $z \in \cal C(\bar \Q_p)$ such that $p^{k-1}z(U_p) = y(U_p)$, $\ell^{k-1}z(T_\ell) = y(T_\ell)$ for all primes $\ell \ndvd N$ and $z(\diamond{a}) = y(\diamond{a})$. For example, if $f \in S_k(\Gamma)$ is in the image of $\theta^{k-1}: M_{2-k}^\dagger(\Gamma) \goto S_k^\dagger(\Gamma)$ then the corresponding point on $\cal C$ has a companion point.

Now let $f$ again be a $p$-ordinary classical newform of level $\Gamma_1(Np^r)$, weight $k\geq 2$ and nebentypus $\ge$. Well known constructions (see \S 2 for details) produce a point $x_{f,\crit} \in \cal C(\bar \Q_p)$ with the following property: $x_{f,\crit}$ arises from a classical eigenform $\twid f \in S_k(\Gamma)$ whose Galois representation $\bar \rho_{\twid f}$ differs from $\bar \rho_f$ by at worst a finite order character and such that $\twid f$ has slope $k-1$ in the sense that $v_p(a_p(\twid f)) = k-1$. The following theorem describes those $f$ for which $\bar \rho_f$ is split at $p$.
\begin{theorem}\label{thm:maintheorem}
Let $f$ be as above. The following are equivalent:
\begin{enumerate}
 \item $\bar \rho_f$ is split at $p$;\label{mainthmcond:split}
 \item $x_{f,\crit}$ has a companion point;\label{mainthmcond:companionpt}
 \item the map $\kappa: \cal C \goto \cal W$ is ramified at $x_{f,\crit}$.\label{mainthmcond:ramified}
\end{enumerate}
\end{theorem}
With the result made precise, let us give a short history of the result and a recollection of previous techniques. First, Theorem \ref{thm:maintheorem} is the characteristic zero analogue of Gross' famous result \cite{Gross-TamenessCriterion} relating the splitting behavior of $\bar \rho_f \bmod p$ to the existence of companion forms modulo $p$ in the sense of Serre. Essentially, a classical eigenform $g$ is produced so that $\theta^{k-1}(g)$ and $f$ have congruent Hecke eigenvalues modulo $p$. We note that the mod $p$ result is separate from ours and cannot be deduced from Theorem \ref{thm:maintheorem}. 

Under some technical hypotheses, progress towards characteristic zero companion forms was made by Ghate \cite{Ghate-OrdinaryForms}, building on the work of Buzzard-Taylor \cite{BuzzardTaylor-CompanionForms}. The techniques there are to lift Gross's theorem to the characteristic zero setting via $\Lambda$-adic families using Mazur's deformation theory of Galois representations, but in a way which is quite different to ours. The first complete proof of the theorem is due to Breuil and Emerton \cite[Theorem 1.1.3]{BreuilEmerton-Ordinary} and uses tools close to those of Gross, but in characteristic zero. Thus, it can be said that in both the mod $p$ or characteristic zero settings one finds that studying $p$-adic cohomology and rigid analytic geometry play a central role. While our proof takes place on a rigid analytic curve, the tools we use are quite different. We proceed in three steps:
\begin{enumerate}[(1)]
\item First, we view $\twid f$ in a rigid analytic family of finite slope eigenforms (the eigencurve $\cal C$) by constructing the point $x_{f,\crit}$.
\item Second, we relate the splitting behavior of the representation $\bar \rho_f$ at $p$ to the (lack of) infinitesimal varation of $\cal C$ relative to $\cal W$ near $x_{f,\crit}$ via the deformation theory of Galois representations.
\item Finally, we explain how the (lack of) infinitesimal variation of the weight is related to companion points via the geometry of the eigencurve over the weight space.
\end{enumerate}
The point (1) is the construction of the eigencurve \cite{ColemanMazur-Eigencurve}, \cite{Buzzard-Eigenvarieties}. If we wish to avoid modular curves and cohomology then we can substitute ideas of Stevens, exposed in \cite[\S 3]{Bellaiche-CriticalpadicLfunctions}, for a construction using overconvergent modular symbols. The key in (2) is essentially a result of Kisin \cite{Kisin-OverconvergentModularForms} on the analytic continuation of crystalline periods and its extension to infinitesimal deformations implicit in \emph{loc. cit.} and explicit in Bella\"iche-Chenevier \cite{BellaicheChenevier-EisensteinSmoothness, BellaicheChenevier-Book}. The point (3) is probably well known, though we make small remarks in the case that $r > 0$.

The aware reader might note that the result in \cite{BreuilEmerton-Ordinary} is slightly more precise than ours in a subtle way which we now will make clear. To fix some terminology, let us say that $\twid f$ has a companion form if there exists a $g$ such that $\theta^{k-1}g = \twid f$. Working with the notation as above, one could consider the following three properties:
\begin{enumerate}
\item $\bar \rho_f$ is split at $p$;
\item $x_{f,\crit}$ has a companion point;
\item $\twid f$ has a companion form.
\end{enumerate}
Among these three, we prove that (a) and (b) are equivalent whereas \cite[Theorem 1.1.3]{BreuilEmerton-Ordinary} says that (a) and (c) are the same. While (c) evidently implies (b), the converse is unclear. However, the \emph{a priori} difference between (b) and (c) is only an issue at primes $\ell \dvd N$ and not important in practice. For example, as far as the application to the $p$-adic local Langlands correspondence in \cite[\S 5]{BreuilEmerton-Ordinary} is concerned, one needs to only know that (b) and (a) are equivalent (see \cite[Proposition 5.4.4]{BreuilEmerton-Ordinary}). One could also directly deduce (c) from (b) if $N = 1$, or by constructing eigencurves using the Hecke operators $(U_\ell)_{\ell \dvd N}$, or by assuming that $H^1_g(G_{\Q},\ad \bar \rho_f) = (0)$ (which is conjectured to always be true) and using the techniques of \cite{Bellaiche-CriticalpadicLfunctions} to see that certain eigenspaces are one-dimensional (see Theorem 4, \emph{loc. cit.}).

Finally, as we mentioned in the previous paragraph, Theorem \ref{thm:maintheorem} has already been applied to local-global compatibility in the $p$-adic local Langlands correspondence for $\GL_2(\Q_p)$. One of the reasons for giving this exposition is that the arguments given here can be generalized to relate the splitting behavior of certain automorphic Galois representations at $p$ to the ramification of definite unitary eigenvarieties over weight spaces, and to the theory of ``companion forms'' \cite{Jones-LocallyAnalytic}, \cite[\S 4]{Chenevier-InfiniteFern} in this setting. This is the focus of future work. It is possible that this, in turn, could play a role in verifying recent conjectures \cite{BreuilHerzig-FundamentalAlgebraic} for local-global compatibility in the (still undefined!) $p$-adic Langlands program for $\GL_n(\Q_p)$.

\section{Proof of Theorem}\label{sec:proof}

\subsection{The eigencurve and construction of $x_{f,\crit}$} Let $N$ be a positive integer and $p$ a prime number such that $p \ndvd N$. We will be interested in the tame level $N$ $p$-adic eigencurve $\cal C(N) = \cal C$. Let $\Gamma := \Gamma_1(N) \intersect \Gamma_0(p)$ and we will recall what $\cal C$ is. First, we denote by $\cal W$ the $p$-adic weight space  $\cal W := \Hom_{\oper{grp-cont}}(\Z_p^\x,\G_m)$. This is a rigid analytic space and the $\C_p$-points are a disjoint union of $p-1$ open unit discs in $\C_p$. Fix an affinoid subdomain $W = \Sp(R) \ci \cal W$ and a real number $\nu\geq 0$. The inclusion $W \ci \cal W$ gives a tautological character $\kappa_R^{\sharp}: \Z_p^\x \goto R^\x$. Coleman \cite{Coleman-ClassicalandOverconvergent} has constructed a finite projective module $M_{R, \leq \nu}^{\dagger}(\Gamma)$ consisting of overconvergent  $p$-adic modular forms of weight $\kappa_R^\sharp$ over $R$ and slope at most $\nu$ for the group $\Gamma$. In the case that $R = L$ is a field then $\kappa^\sharp = \kappa_L^\sharp: \Z_p^\x \goto L^\x$ is a Serre weight and we will denote the module $M_{L,\leq \nu}^\dagger(\Gamma)$ by $M_{\kappa^\sharp,\leq \nu}^\dagger(\Gamma)$. There are also cuspidal version $S^\dagger$ of each of the modules $M^\dagger$.

Let $\bf T_{R,\leq \nu}$ be the subalgebra of $\End_R(M^\dagger_{R,\leq \nu}(\Gamma))$ generated by the image of the Hecke algebra $\cal H = \Z[(T_\ell)_{\ell \ndvd Np},U_p,\diamond{a}]$ for the group $\Gamma$. This is an finite $R$-algebra (thus a $\Q_p$-affinoid) and we let $\cal C_{W,\leq \nu} = \Sp(\bf T_{R,\leq \nu})$ be the associated rigid analytic space in the sense of Tate.  The points are in bijection with the closed points of $\Spec(\bf T_{R,\leq\nu})$, though we give an explicit description of some points below. We remark as well that $\cal C_{W,\leq \nu}$ is reduced (see \cite[Theorem 3.30]{Bellaiche-CriticalpadicLfunctions}). The structure map $R \goto \bf T_{R,\leq \nu}$ defines a map $\kappa: \cal C_{W,\leq \nu} \goto W$ which we call the weight map. As pointed out earlier, $\kappa$ is finite and flat. 

The eigencurve $\cal C$ together with the weight map $\kappa: \cal C\goto \cal W$ is constructed via a gluing process which is quite delicate in general. The case $N = 1$ is is carried out in \cite{ColemanMazur-Eigencurve} and the case $N > 1$  in \cite{Buzzard-Eigenvarieties}. Since our questions here are local in nature on $\cal C$ (and $\cal W$) we will content ourselves with the above summary. However, to make notation easier and supress inconsequential choices, we work with the global object $\cal C$. Replacing $M^\dagger$ by the cuspidal version $S^\dagger$ one could also construct local spaces $\cal C^0_{W,\leq \nu}$ whose points correspond to systems of Hecke eigenvalues appearing on the modules $S^{\dagger}_{R,\leq \nu}(\Gamma)$ of cuspidal overconvergent modular forms. Let $\cal C^0 \ci \cal C$ be the corresponding global object. To make precise our normalization in all of this, let us declare that if $x \in \cal C(L)$ is a system of eigenvalues appearing in a space of cuspforms $S_k^\dagger(\Gamma)$ then $\kappa(x) = z^{k} \in \Z \ci \cal W(L)$. With this normalization one can think of the element in the weight space as corresponding, up to a shift by $-1$, to the Hodge-Tate weight of the point (cf. Lemma \ref{lemma:wtvssenwt}). We have the following explicit description of the local cuspidal pieces of the eigencurve.

\begin{proposition}\label{thm:eigencurve}\label{prop:eigencurve}
Let $L$ be a field and $\kappa^\sharp \in W(L)$. There is a bijection
\begin{equation*}
\xymatrix{
x \ar@{|->}[d] & \kappa^{-1}(\kappa^{\sharp}) = \set{x \in \cal C^0_{W,\leq \nu}(L) \st \kappa(x) = \kappa^\sharp}\ar@{<->}[d]\\
T\mapsto T(x):\cal H \goto L & \set{\text{systems of eigenvalues of $\cal H$ appearing in $S_{\kappa^\sharp,\leq \nu}^{\dagger}(L)$}}.
}
\end{equation*}
\end{proposition}

We now recall the construction of many points attached to classical modular forms. We include the proof of the following proposition for the convenience of the reader since it includes the specification of the form $\twid f$ and thus what will become the point $x_{f,\crit}$.

\begin{proposition}\label{prop:classicalconstruction}
Fix a newform $f \in S_k(\Gamma_1(Np^r),\ge)$. Write the $q$-expansion at infinity as $f = \sum_{n\geq 1} a_n(f) q^n$. Suppose that $a_p(f) \neq 0$ and $r = v_p(\cond \ge)$. Decompose $\ge=\ge_p\ge^p$ into its $p$-part $\ge_p$ and its prime-to-$p$ part $\ge^p$. Then, there exists two points $x_{1,f}$ and $x_{2,f}$ on $\cal C^0$ such that
\begin{equation*}
 T_\ell(x_{1,f}) = T_\ell(x_{2,f})\ge_p(\ell) \;\;\;\; \text{(if $\ell \ndvd Np$)}\\
\end{equation*}
Furthermore,
\begin{enumerate}
 \item If $r  = 0$ then $\set{U_p(x_{1,f}),U_p(x_{2,f})}$ are the two roots of the Hecke polynomial $X^2 - a_p(f)X + p^{k-1}\ge(p)$ attached to $f$.
 \item If $r > 0$ then $U_p(x_{1,f}) = a_p(f)$ and $U_p(x_{2,f}) = p^{k-1}\ge^p(p)/a_p(f)$.
\end{enumerate}
In any case, $U_p(x_{1,f})U_p(x_{2,f}) = p^{k-1}\ge^p(p)$.
\end{proposition}
\begin{proof}
 The case of $r = 0$ goes back at least to \cite{Mazur-padicvariation}. Let $X^2 - a_pX + p^{k-1}\ge(p)$ be the characteristic polynomial for the operator $T_p$. It has two roots $\alpha$ and $\beta$ in $\bar \Q$. Then, the two forms
\begin{align*}
f_\alpha(z) &:= f(z) - \beta f(pz)\\
f_\beta(z) &:= f(z) - \alpha f(pz)
\end{align*}
are eigenforms on the group $\Gamma$. One immediately checks that system of eigenvalues determined by $f_\alpha$ and $f_\beta$ is the same as $f$ away from $p$. At $p$,  $U_p f_\alpha = \alpha f_\alpha$ and $U_p f_\beta = \beta f_\beta$. The nebentypus of $f$, which is trivial at $p$, remains unchanged. These are our two points $x_{1,f}$ and $x_{2,f}$. 

Now assume that $r > 0$. View $\ge_p$ as a character of $\Z_p^\x$ and let $\kappa^\sharp(z) = z^{k}\ge_p(z)$ for $z \in \Z_p^\x$. Then we have an inclusion $S_k(\Gamma_1(Np^r),\ge) \ci S_{\kappa^\sharp}^\dagger(\Gamma_0(p^r)\intersect \Gamma_1(N),\ge^p)$ (notice the change in the weight). The theory of the canonical subgroup may be used to make an identification between spaces of Coleman's overconvergent forms
\begin{equation*}
S^\dagger_{\kappa^\sharp}(\Gamma_0(p^r)\intersect \Gamma_1(N))^{\text{fin.slope}} = S^\dagger_{\kappa^\sharp}(\Gamma)^{\text{fin.slope}}.
\end{equation*}
Thus, we see that the form $f$ produces a finite slope system of eigenvalues for the Hecke algebra acting on the space $S^\dagger_{\kappa^\sharp}(\Gamma)$. Call the corresponding point on $\cal C^0$ the point $x_{1,f}$. 

To produce the other point $x_{2,f}$ we use a technique we learned from \cite[\S 6]{Ghate-OrdinaryForms} (but see also \cite[\S 4]{BreuilEmerton-Ordinary} for the representation-theoretic description). The point is that for an integer $k$ there is an isomorphism
\begin{equation*}
w_{p^r}: S_k(\Gamma_1(Np^r),\ge) \overto{\iso} S_k(\Gamma_1(Np^r),\ge^p\ge_p^{-1})
\end{equation*}
given by an Atkin-Lehner operator $w_{p^r}$. It preserves newforms up to constants and the normalized newform $\twid f$ on the right hand side corresponding to $f$ has the system of Hecke eigenvalues
\begin{equation*}
 a_\ell(\twid f) = \begin{cases}
 	          p^{k-1}\ge^p(p)/a_p(f) & \text{if $\ell = p$},\\
               a_\ell(f)\ge_p^{-1}(\ell) & \text{if $\ell \neq p$.}
              \end{cases}
\end{equation*}
Here we used our assumption that $\ge$ was primitive modulo $p^r$. The system of eigenvalues corresponding to $\twid f$ then produces a point $x_{1,\twid f} =: x_{2,f}$ by the previous paragraph. The identities in the statement of the proposition follow.
\end{proof}
The construction has the following consequences for Galois representations. By constructions of Eichler-Shimura and Deligne we know that to a newform $f \in S_k(\Gamma_1(Np^r),\ge)$ we can associate a Galois representation $\bar \rho_f: G_{\Q,Np} \goto \GL_2(\bar \Q_p)$ such that if $\ell \ndvd Np$ then $\bar \rho_f$ is unramified at $\ell$ and $\bar \rho_f(\Frob_\ell)$ has characteristic polynomial $X^2 - a_\ell(f)X + \ell^{k-1}\ge(\ell)$. Below we describe how to interpolate these representations to get representations $\bar\rho_x$ attached to each point $x \in \cal C(\bar \Q_p)$ of the eigencurve. By looking at the first statement made in the above proposition we immediately deduce that $\bar \rho_{x_{1,f}}$ and $\bar \rho_{x_{2,f}}$ are equal up to a twist by a finite order character. In the case where $f$ has no level at $p$, $\bar \rho_{x_{1,f}} = \bar \rho_{x_{2,f}}$ on the nose.

Now suppose that $f$ is ordinary at $p$ and $r = v_p(\cond \ge)$. By definition then, $v_p(a_p) = 0$. Using the final statement of Proposition \ref{prop:classicalconstruction} we have points $x_{i,f}$ such that $v_p(U_p(x_{1,f})) + v_p(U_p(x_{2,f})) = k-1$. Without loss of generality $v_p(U_p(x_{1,f})) = 0$ (Proposition \ref{prop:classicalconstruction}(b) forces this if $r > 0$). Thus, $v_p(U_p(x_{2,f})) = k-1$ and we say $x_{2,f}$ is a point on $\cal C^0$ of \emph{critical slope}. In the notation of the proof of Proposition \ref{prop:classicalconstruction} we have that $\twid f = f_\beta$ in the case $r = 0$.

\begin{definition}
If $f$ is ordinary at $p$ and $r = v_p(\cond \ge)$ let $x_{f,\crit} \in \cal C^0$ be the point $x_{2,f}$ of critical slope constructed above.
\end{definition}

We now have all the information we need to understand Theorem \ref{thm:maintheorem}. We remark immediately that we've only constructed/defined the point $x_{f,\crit}$ for the $f$ above satisfying that $a_p \neq 0$, and if $r \geq 1$ then we had a condition the conductor of $\ge$. The ordinariness of the theorem subsumes the condition that $a_p \neq 0$. However, one could also have an ordinary form $f$ which satisfies that $r = 1$ but $\ge = \ge^p$. In this case, $\bar \rho_{f,p}$ will never be split and $x_{f,\crit}$ will never have a companion point so the theorem is a tautology (see the discussion in \cite[\S 6]{Ghate-OrdinaryForms}).

\subsection{Infinitesimal variation of the weight}
Up until now our setup is inspired heavily by the setups in \cite{Ghate-OrdinaryForms} and \cite{BreuilEmerton-Ordinary}. We've constructed $\twid f$ as they did but our departure point comes in the proof of Theorem \ref{thm:maintheorem}. It will follow from considering the infinitesimal deformation theory of the refined family of Galois representations on $\cal C$ in the sense of \cite[Ch. 4]{BellaicheChenevier-Book}. We quickly recall the main points. 

For now we let $x \in \cal C^0(\bar \Q_p)$ be a very\footnote{This terminology is due to Bella\" iche and will appear in the forthcoming book \cite{Bellaiche-EigenvarietyCourse}. The other points constructed in Proposition \ref{prop:classicalconstruction} are called Hida classical points in \emph{loc. cit}. Together they make up all the classical points.} classical point in the sense that $x$ is one of the points $\set{x_{i,f}}_{i=1,2}$ constructed in the $r = 0$ case of Proposition \ref{prop:classicalconstruction}. We denote by $\cal Z$ all of the very classical points. Let $\bar \rho_x$ be the Galois representation $\bar \rho_f = \bar \rho_x: G_{\Q,Np} \goto \GL_2(\bar \Q_p)$ attached to $f$. This is independent of $x = x_{i,f}$ by the remarks proceeding Proposition \ref{prop:classicalconstruction}. Since the levels are prime to $p$, all the representations $\set{\bar \rho_{x}}_{x \in \cal Z(\bar \Q_p)}$ are crystalline at $p$. Each of the representations $\bar \rho_x$ is absolutely irreducible. Furthermore, it is well known that $D^+_{\cris}(\bar \rho_x)^{\varphi = U_p(x)} \neq 0$ (see \cite[Theorem 1.2.4(ii)]{Scholl-MotivesForModularForms}).

For any such $x$ we have a Galois pseudocharacter $T_x: G_{\Q,Np} \goto \bar \Q_p$ given by $T_x(g) = \tr(\bar \rho_x(g))$. It follows from \cite[Prop. 7.1.1]{Chenevier-padicAutomorphicForm} that there is a unique pseudocharacter $T: G_{\Q,Np} \goto \cal O(\cal C)$ such that for any $x$ as before we have $\ev_x \compose T = T_x$. In particular, moving to a point $x \nin \cal Z$ then a theorem of Taylor \cite[Theorem 1(2)]{Taylor-Pseudoreps} gives us a unique semi-simple representation $\bar \rho_x: G_{\Q,Np} \goto \GL_2(\bar \Q_p)$ such that $\tr(\bar \rho_x) = \ev_x\compose T$. We refer to the representations $\set{\bar \rho_x}$, or equivalently the pseudocharacter $T$, as the family of Galois representations on $\cal C$. We pause here to describe the only ingredient left in seeing the eigencurve as a refined family in the sense of \cite[Ch. 4]{BellaicheChenevier-Book}. By the previous constructions and Sen's theory in families, there is an analytic function  $s: \cal C \goto \Af^1$ so that at a point $x \in \cal C(\bar \Q_p)$ the $p$-adic Sen weights of the representation $\bar \rho_x$ are given by $0$ and $s(x)$.

\begin{lemma}\label{lemma:wtvssenwt}
Let $\kappa: \cal C \goto \cal W$ be the weight map and $s: \cal C \goto \Af^1$ be the Sen map. We have a commuting diagram
\begin{equation*}
\xymatrix{
& \cal C  \ar[dl]_-{\kappa} \ar[dr]^-s \\
\cal W \ar[rr]_-{-1 + \log} & & \Af^1
}
\end{equation*}
In particular, if $x \in \cal C(\bar \Q_p)$ then the induced maps $d\kappa, ds : T_x\cal C \goto L(x)$ on tangent spaces are either both zero or both isomorphisms.
\end{lemma}
\begin{proof}
By \cite[Lemma 7.5.12]{BellaicheChenevier-Book} it is only necessary to check this on $\cal Z$. There, though, we know that the weight map is normalized so that $x \in \cal Z$ is mapped to the character $z \mapsto z^{k}$. On the other hand, the Sen weights of $\bar \rho_x$ are $0$ and $k-1$.
\end{proof}

It is worth making a comment about the points $x_{i,f}$ of Proposition \ref{prop:classicalconstruction} where $r \geq 1$ since we excluded these points in $\cal Z$. So, let $f$ be a cuspidal newform of level $\Gamma_1(Np^r)$ with character $\ge = \ge_p\ge^p$. In the proof of Proposition \ref{prop:classicalconstruction} we found the system of Hecke eigenvalues attached to $f$ in the Hecke module $S^\dagger_{\kappa^\sharp}(\Gamma)$ where $\kappa^\sharp = z^{k}\ge_p$. The proof of the lemma in this case is that since $\ge_p$ is a finite order character, it vanishes under the logarithm. So, we see then that such points correspond to elements in $\cal C^0(\bar \Q_p)$ whose weights are not of the form $z^{\N}$ but whose Sen weights are integers.

Following the lemma, we can say that $(T,\kappa,\cal Z, (U_p,U_p^{-1}))$ gives the data of a family of refined Galois representations on $\cal C$ in the sense of \cite[Ch. 4]{BellaicheChenevier-Book}. In particular, their extension \cite[Lemma 6]{BellaicheChenevier-EisensteinSmoothness}, \cite[Theorem 4.3.2]{BellaicheChenevier-Book} of Kisin's result \cite[Theorem 6.3]{Kisin-OverconvergentModularForms} applies to our situation and we have the full analytic continuation of crystalline periods, which we now recall.

From now on, we let $x$ be a point corresponding to a classical cuspidal newform of weight $k$ (though not necessarily the $r = 0$ case). We will eventually apply what we say to $x = x_{f,\crit}$ and we invite the the reader to make that specializiation now. Since $k\neq 1$, the Sen weights of $\bar \rho_x$ are the distinct integers $0 < k-1$. We denote by $A_x := \cal O_{\cal C,x}$ the local ring of $\cal C$ at $x$. Its residue field we denote by $L(x)$. Since $\bar \rho_x$ is absolutely irreducible and $A_x$ is Henselian, the theorem of Nyssen-Rouquier \cite[Corollarie 5.2]{Rouqier-Jalgebra96-Pseudocharacters} implies that there is a unique two dimensional semisimple representation $\rho_x: G_{\Q,Np} \goto \GL_2(A_x)$ such that $\rho_x \tensor_{A_x} L(x) = \bar \rho_x$. If $I \ci A_x$ is an ideal of cofinite length then the analytic continuation of crystalline periods implies that $D_{\cris}^+(\restrict{\rho_x}{G_{\Q_p}} \tensor_{A_x} A_x/I)^{\varphi = U_p}$ is free of rank one over $A_x/I$.

If $\fr{AR}_{L(x)}$ is the category of local Artin $L(x)$-algebras with residue field $L(x)$ then we let $X_{\bar \rho_x}: \fr{AR}_{L(x)} \goto \Set$ be the formal deformation functor in the sense of \cite{Mazur-MSRI89-Deformations} for $\bar \rho_x$. Since $\bar \rho_x$ is absolutely irreducible, this functor is pro-representable by a complete local noetherian ring $R_x$ with residue field $L(x)$. We have as well the functor $X_{\bar \rho_{x,p}}$ where $\bar \rho_{x,p} := \restrict{\bar \rho_x}{G_{\Q_p}}$. Notice that in the situation of Theorem \ref{thm:maintheorem}, $\bar \rho_{x,p}$ will be split and thus $X_{\bar \rho_{x,p}}$ will not be representable. Nevertheless, we consider the subfunctor $X_{\bar \rho_{x,p}}^{U_p(x)} \ci X_{\bar \rho_{x,p}}$ first appearing in \cite[\S 8]{Kisin-OverconvergentModularForms}, defined as
\begin{equation*}
X_{\bar \rho_{x,p}}^{U_p(x)} = \set{\rho_A \in X_{\bar \rho_x,p}(A) \st D_{\cris}^+(\rho)^{\varphi = U} \text{ is free of rank one for some $U \congruent U_p(x) \bmod \ideal m_A$}}.
\end{equation*}
By \cite[Prop. 8.13]{Kisin-OverconvergentModularForms} the functor $X_{\bar \rho_{x,p}}^{U_p(x)}$ is relatively representable over $X_{\bar \rho_{x,p}}$ and thus the fibered product
\begin{equation*}
X_{\bar \rho_x}^{U_p(x)} := X_{\bar \rho_x} \x_{X_{\bar \rho_{x,p}}} X_{\bar \rho_{x,p}}^{U_p(x)}
\end{equation*}
is pro-representable by a complete local noetherian ring $R_x^{U_p(x)}$ with residue field $L(x)$ and which is a quotient of $R_x$.

The representation $\rho_x$ appearing on the local ring of the eigencurve defines a map $R_x \goto A_x$ and by the theorem on analytic conintuation of crystalline periods it must factor through the quotient $R_{x}^{U_p(x)}$. Since $A_x$ is generated by the Hecke operators (including $U_p$) and the weight map, $A_x$ is a quotient of $R_x^{U_p(x)}$.

Let $T_x$ denote the Zariski tangent space $X_{\bar \rho_x}(L(x)[\ge])$ associated to the deformation functor $\bar \rho_x$. It has a subspace $T_x^{U_p(x)} = X^{U_p(x)}_{\bar \rho_x}(L(x)[\ge])$ corrsponding to the functor $X_{\bar \rho_x}^{U_p(x)}$. If $\twid x \in T_x$ is a point we let $\rho_{\twid x}: G_{\Q,Np} \goto \GL_2(L(x)[\ge])$ be the corresponding deformation. The following proposition details the infinitesimal variation of at least one of the weights.

\begin{proposition}\label{prop:constantweightprop}
Suppose that the line $D_{\cris}^+(\bar \rho_x)^{\varphi = U_p(x)} \ci D_{\cris}^+(\bar \rho_x)$ has induced Hodge-Tate weight $b \geq 0$. Then $b$ is a constant Sen weight of $\rho_{\twid x}$ for any $\twid x \in T_x^{U_p(x)}$. 
\end{proposition}

\begin{remark}
Notice that \emph{a priori} $\rho_{\twid x}$ has a Sen weight of the form $b  + b(\twid x)\ge$. By constant weight we mean that $b(\twid x) = 0$.
\end{remark}
\begin{proof}
Suppose that $D_{\cris}^+(\bar \rho_{x,p})^{\varphi = U_p(x)}$ has Hodge-Tate weight 0. In this case, it follows from \cite[Proposition 2.5.4]{BellaicheChenevier-Book} that for any $A$, any deformation in $X_{\bar \rho_x}^{U_p(x)}(A)$ will have a Sen weight zero.

If $b \neq 0$ then $\bar \rho_{x,p}$ must be split because it is potentially semi-stable. Unlike before, we really must work only at the level of tangent spaces. The proof follows the key compuation in \cite[Theorem 2.6]{Bellaiche-CriticalpadicLfunctions}, but we include it for the convenience of the reader. The Zariski $T_x$ tangent space is well known to be identified with the cohomology space $H^1(G_{\Q,Np},\ad \bar \rho_x)$ and the restriction map defines a morphism $T_x \goto H^1(G_{\Q_p},\ad \rho_{x,p})$. Since the Sen weights only depend on $\rho_{\twid x, p}$ it suffices to consider its image there.

Let us write $\bar \rho_{x,p} = \chi_1 \dsum \chi_2$ with $\chi_2$ crystalline of weight $b$ and $D_{\cris}^+(\chi_2)^{\varphi = U_p(x)} \neq 0$. We have a decomposition
\begin{equation*}
H^1(G_{\Q_p},\ad \bar \rho_{x,p}) = \bigdsum_{i,j=1}^2 H^1(G_{\Q_p},\chi_i\chi_j^{-1}) = \bigdsum_{i,j=1}^2 \Ext^1_{L(x)[G_{\Q_p}]}(\chi_i,\chi_j).
\end{equation*}
On the other hand, if $D_{\sen}$ is Sen's functor then we have 
\begin{equation*}
\xymatrix{
H^1(G_{\Q_p},\ad \bar \rho_{x,p}) \ar@{=}[d] \ar[r]^-{D_{\sen}} & H^1(\Gamma_\infty,\ad D_{\sen}(\bar \rho_{x,p})) \ar@{=}[d]\\
\bigdsum_{i,j=1}^2 \Ext^1_{L(x)[G_{\Q_p}]}(\chi_i,\chi_j) \ar[r]^-{\dsum D_{\sen}} & \bigdsum_{i,j=1}^2 \Ext^1_{L(x)[\Gamma_\infty]}(D_{\sen}(\chi_i),D_{\sen}(\chi_j))
}
\end{equation*}
where $\Gamma_\infty$ is the $p$-cyclotomic quotient of $G_{\Q_p}$. Since $\chi_1$ and $\chi_2$ have distinct weights, an easy computation shows that if $i \neq j$ then the $(i,j)$ term in the lower right sum disappears. Furthermore the weight of $\chi_i$ is constant in a deformation $\rho_{\twid x,p}$ if and only if its image in $\Ext^1_{L(x)[\Gamma_\infty]}(D_{\sen}(\chi_i),D_{\sen}(\chi_i))$ is zero. We now proceed to show this for $i = 2$.

We consider the left vertical equality above. If $\twid x \in T_x^{U_p(x)}$ then it follows that the image of $\twid x$ under the map
\begin{equation*}
H^1(G_{\Q_p},\ad \bar \rho_{x,p}) \goto \Ext^1_{L(x)[G_{\Q_p}]}(\chi_2,\chi_2)
\end{equation*}
lands inside the subspace $\Ext^{1}_{L(x)[G_{\Q_p}],\cris}(\chi_2,\chi_2)$ consisting of cystalline $L(x)[\ge]$-valued characters deforming $\chi_2$. In particular, since such a character is Hodge-Tate it vanishes under the composition of the bottom horizontal map. Thus we have shown our claim.
\end{proof}

\subsection{Ramification above the weight space and companion points}

The final ingredient in the proof of Theorem \ref{thm:maintheorem} is the relationship between between Coleman's $\theta$-operator and companion points. Recall that while the operator $\theta:= q\der /q$ acting on $p$-adic modular forms will in general not preserve overconvergence it does have the property that $\theta^{k-1}$ will take a finite slope overconvergent form of weight $2-k$ to a finite slope overconvergent form of weight $k$. Relevant for us is Coleman's study \cite{Coleman-ClassicalandOverconvergent, Coleman-OverconvergentModularFormsOfHigherLevel} of the relationship between the cokernel $M_k^\dagger(\Gamma_1(Np^r))/\theta^{k-1}M^\dagger_{2-k}(\Gamma_1(Np^r))$  and spaces of classical modular forms. 

If $r \geq 1$ denote by $S_k^{\prim}(\Gamma_1(Np^r))$ the subspace of $S_k(\Gamma_1(Np^r))$ spanned by forms of level $\Gamma_1(Np^{r'})$ with $r' \leq r$ and whose nebentypus at $p$ has conductor exactly $p^{r'}$. If $r = 0$ we define $S_k^{\prim}(\Gamma_1(N))$ to be the subspace of $S_k(\Gamma_1(Np))$ spanned by the $p$-old forms. Following Coleman, we denote by $S_k^{\dagger,0}(\Gamma_1(Np^r))$ the subspace of $S_k^{\dagger}(\Gamma_1(Np^r))$ consisting of overconvergent forms which have trivial residues on each supersingular annulus. If $V$ is a vector space together with a linear action of $U_p$, we let $V^{\nu = a}$ denote the $U_p$-slope $a$ subspace.

\begin{proposition}\label{prop:control}
For any $0\leq a \leq k-1$ there is a short exact sequence of Hecke modules
\begin{equation*}
0 \goto \left( M_{2-k}^{\dagger}(\Gamma_1(Np^r))({\det}^{k-1})\right)^{\nu = a}\overto{\theta^{k-1}}S_k^{\dagger,0}(\Gamma_1(Np^r))^{\nu = a} \goto S_k^{\prim}(\Gamma_1(Np^r))^{\nu = a} \goto 0.
\end{equation*}
\end{proposition}
\begin{proof}
If $a < k-1$ then this follows from \cite{Coleman-ClassicalandOverconvergent} for $r = 0,1$ and \cite{Coleman-OverconvergentModularFormsOfHigherLevel} for $r \geq 2$. If $r = 0$ then the boundary case $a = k-1$ is \cite[Corollary 7.2.2]{Coleman-ClassicalandOverconvergent} and Coleman's argument there may be adapted to handle $r \geq 1$. For simplicity we only sketch an argument when $r = 1$ and use \cite[\S 8]{Coleman-ClassicalandOverconvergent} for notation and  reference. The modifications for $r \geq 2$ are made by making apparent adjustments and using \cite{Coleman-OverconvergentModularFormsOfHigherLevel}.

We now impose the notation of \cite[\S 8]{Coleman-ClassicalandOverconvergent}. In particular, we have the parabolic cohomology space $H_{\parab}(k-2,p)$ which, by \cite[Proposition 8.2]{Coleman-ClassicalandOverconvergent}, is realized as the cokernel
\begin{equation*}
S_k^{\dagger,0}(\Gamma_1(Np))/\theta^{k-1}M_{2-k}^{\dagger}(\Gamma_1(Np)) \iso H_{\parab}(k-2,p).
\end{equation*}
Our goal is to show that $H_{\parab}(k-2,p)$ is isomorphic to $S_{k}^{\prim}(\Gamma_1(Np))$ as Hecke-modules. We denote by $w_N, w_p$ and $w_{Np}$ the Atkin-Lehner operators, depending on fixed roots of unity, on the modular curve $X_1(Np)$ ($w_p$ already appeared in the proof of Proposition \ref{prop:classicalconstruction}).

The space $H_{\parab}(k-2,p)$ is equipped with a natural pairing $(\;,\;)_k$ (denoted $(\;,\;)_k^{\infty}$ by Coleman) and the adjoints of the Hecke operators at level $Np$ are
\begin{align}\label{eqn:pairing-props}
(\restrict{x}{\langle \ell \rangle},y)_k &= (x,\restrict{y}{\langle \ell \rangle^{-1}})_k,\\
\nonumber (\restrict{x}{T_\ell} , y)_k &= (x,\restrict{y}{\langle \ell \rangle^{-1} T_\ell})_k, \text{ and}\\
\nonumber (\restrict{x}{U_p} , y)_k &= (x,\restrict{y}{w_{Np}^{-1}U_pw_{Np}})_k
\end{align}
(\cite[page 32]{Coleman-padicShimura} is a convenient reference). The dimensions of $H_{\parab}(k-2,p)$ and $S_k^{\prim}(\Gamma_1(Np))$ are the same by \cite[Proposition 8.4]{Coleman-ClassicalandOverconvergent} and, moreover, the natural map $\iota: S_k^{\prim}(\Gamma_1(Np)) \goto H_{\parab}(k-2,p)$ is an isomorphism in slope $a < k-1$ \cite[Lemma 8.7]{Coleman-ClassicalandOverconvergent}. Finally, there are decompositions
\begin{align*}
H_{\parab}(k-2,p) &= H_{\parab}(k-2) \dsum H_{\parab}^{\oper{new}}(k-2,p)\\
S_k^{\prim}(\Gamma_1(Np)) &= S_k^{\prim}(\Gamma_1(N)) \dsum S_k^{\oper{new}}(\Gamma_1(Np))
\end{align*}
and \cite[Theorem 7.2]{Coleman-ClassicalandOverconvergent} says that $H_{\parab}(k-2)$ is isomorphic to  $S_k^{\prim}(\Gamma_1(N))$ as Hecke-modules. Thus, it suffices to focus on the slope $a = k-1$ subspaces of the new components. Note that the Atkin-Lehner operator $w_p$ induces an isomorphism 
\begin{equation*}
w_p : S_k^{\oper{new}}(\Gamma_1(Np))^{\nu=k-1} \overto{\iso} S_k^{\oper{new}}(\Gamma_1(Np))^{\nu=0}.
\end{equation*}
This will play the role of the isomorphism $r$ from \cite[p. 223]{Coleman-ClassicalandOverconvergent}. If we denote by $U_p^{\ast}$ the conjugate $w_p^{-1}U_pw_p$ then $U_p^{\ast}U_p=U_pU_p^{\ast} = \langle p \rangle_N p^{k-1}$ on the new subspace $S_k^{\oper{new}}(\Gamma_1(Np))$. In particular, $w_{Np}^{-1}U_pw_{Np} = p^{k-1}U_p^{-1}$ as well. Using that $\iota$ is an isomorphism in slope less than $k-1$, the formulas \eqref{eqn:pairing-props}, and the formula for $U_p^{\ast}U_p$, one checks the pairing 
\begin{equation*}
[x,y] := (x,\restrict{\iota(\restrict{y}{w_p})}{w_N})_k
\end{equation*}
is a Hecke-equivariant, perfect pairing between the spaces $H_{\parab}^{\oper{new}}(k-2,p)^{\nu = k-1}$ and $S_k(\Gamma_1(Np))^{\oper{new},\nu = k-1}$, generalizing \cite[Lemma 7.3]{Coleman-ClassicalandOverconvergent}.
\end{proof}

We return to studying companion points.  If $x: \cal H \goto \bar \Q_p$ is a finite slope system of eigenvalues and $M$ is a Hecke module we use $M_{(x)}$ to denote the generalized eigenspace for $x$. In this notation, we recall that $x$ defines a point on the eigencurve if and only if $M^\dagger_{\kappa^\sharp}(\Gamma)_{(x)} \neq 0$ for some $\kappa^\sharp \in \cal W$.

\begin{corollary}\label{cor:jump-dim}
Let $f$ be a $p$-ordinary newform of level $\Gamma_1(Np^r)$, weight $k\geq 2$ and nebentypus $\ge$ where $v_p(\cond \ge) = r$. The critical slope point $x_{f,\crit}$ has a companion point if and only if $\dim S^\dagger_{\kappa(x_{f,\crit})}(\Gamma)_{(x_{f,\crit})} > 1$.
\end{corollary}
\begin{proof}
Recall, as was mentioned in the proof of Proposition  \ref{prop:classicalconstruction}, that we have identifications
\begin{equation*}
S_k^\dagger(\Gamma_1(Np^r),\ge)^{\text{fin.slope}} = S_{\ge_p\tensor k}^\dagger(\Gamma_0(p^r) \intersect \Gamma_1(N), \ge^p)^{\text{fin.slope}} = S_{\ge_p\tensor k}^\dagger(\Gamma,\ge^p)^{\text{fin.slope}}.
\end{equation*}
Under these identifications, Proposition \ref{prop:control} shows that we have an exact sequence
 \begin{equation*}
  0 \goto M^\dagger_{\ge_p\tensor 2-k}(\Gamma)({\det}^{k-1})_{(x_{f,\crit})} \goto S^\dagger_{\ge_p\tensor k}(\Gamma)_{(x_{f,\crit})} \goto S_{k}(\Gamma)_{(x_{f,\crit})} \goto 0.
 \end{equation*}
 By the theory of newforms, the quotient is one-dimensional. On the other hand, $x_{f,\crit}$ has a companion point if and only if the subobject in the sequence is non-zero.
\end{proof}

Following this result, we can now give the proof of Theorem \ref{thm:maintheorem}.

\begin{proof}[{Proof of Theorem \ref{thm:maintheorem}}]
Let $f$ be as in the theorem and let $x = x_{f,\crit}$ be the associated critical slope point on $\cal C$. Since $\bar \rho_x = \bar \rho_f$ up to a twist we can replace $\bar \rho_f$ by $\bar \rho_x$ (recall the remarks proceeding Proposition \ref{prop:classicalconstruction}). 

The first step is to note that by \cite[Lemma 2.8(iii)]{Bellaiche-CriticalpadicLfunctions}, the length of the scheme-theoretic fiber of $\kappa$ at $x$ is the same as the dimension of $S_{\kappa(x)}^{\dagger}(\Gamma)_{(x)}$. Since the weight map is locally flat, we get that $\kappa$ is ramified at $x$ if and only if $\dim S_{\kappa(x)}^\dagger(\Gamma)_{(x)} > 1$. By Corollary \ref{cor:jump-dim} we see that (b) and (c) are equivalent. We will now show that (b) implies (a) and that (a) implies (c).

Suppose first that $x$ has a companion point $y$. The point $y$ corresponds to an overconvergent $p$-adic modular form $g$ and let $g' := \theta^{k-1}(g)$. It is well known, for example by \cite[Theorem 6.6(ii)]{Kisin-OverconvergentModularForms} or \cite[Proposition 11]{Ghate-OrdinaryForms} , that this implies that $\bar \rho_{g'}$ is split. But, the three representations $\bar \rho_{g'}$,   $\bar \rho_g$ and $\bar \rho_x$ are all equal up to at worst a twist. The first two are because $\theta^{k-1}$ only twists the representation and the first and third because they have the same Frobenius eigenvalues away from $Np$. Thus $\bar \rho_x$ is split at $p$ and (b) implies (a).

Now suppose that $\bar \rho_x$ is split and we will show that $\kappa$ is ramified at $x$. For this is enough to test whether or not the differential $d\kappa: T_x\cal C \goto T_w\cal W$ of the weight map is zero.  To do this, we are going to make auxillary use of the abstract Galois deformation theory. Consider the Zariski tangent space $T_x^{U_p(x)}$ associated to the deformation problem $X_{\bar \rho_x}^{U_p(x)}$ as in Proposition \ref{prop:constantweightprop}. The surjection $R_x^{U_p(x)} \surject A_x$ induces an inclusion $\iota: T_x\cal C \inject T_x^{U_p(x)}$. Let $\twid x \in T_x\cal C$. By Lemma \ref{lemma:wtvssenwt} we have a way to compute $d\kappa(\twid x)$. What we do is look at $\iota(\twid x) = \rho_{\twid x}$, which is a deformation of $\bar \rho_x$ to $L(x)[\ge]$. We then compute its two Sen weights $a(\twid x)\ge$ and $k-1 + b(\twid x)\ge$. The map to the weight space is $d\kappa(\twid x) = b(\twid x)$. Since $x$ has slope $k-1$ and $\bar \rho_x$ is split, the Hodge-Tate weight of the line $D_{\cris}^+(\bar \rho_x)^{\varphi=U_p(x)}$ is $k-1$. So, we can apply Proposition \ref{prop:constantweightprop} with $b = k-1$. By that lemma, we have that every element of $T_x^{U_p(x)}$ has constant weight $k-1$ and thus $b(\twid x) = 0$.
\end{proof}

\bibliography{../../../bibliography/master}

\begin{thebibliography}{Maz2}

\bibitem[Bel1]{Bellaiche-EigenvarietyCourse}
Bella{\"{\i}}che, J.
\newblock {\em Eigenvarieties, families of Galois representations, $p$-adic
  $L$-functions}.
\newblock Available at \verb|http://people.brandeis.edu/~jbellaic|, 2011.

\bibitem[Bel2]{Bellaiche-CriticalpadicLfunctions}
Bella{\"{\i}}che, J.
\newblock {Critical $p$-adic {$L$}-functions}.
\newblock {\em Invent. Math.} {\textbf{189}}(2012), 1--60.

\bibitem[BC1]{BellaicheChenevier-EisensteinSmoothness}
Bella{\"{\i}}che, J. and Chenevier, G.
\newblock {Lissit\'e de la courbe de {H}ecke de {$\rm GL\sb 2$} aux points
  {E}isenstein critiques}.
\newblock {\em J. Inst. Math. Jussieu} {\textbf{5}}(2006), 333--349.

\bibitem[BC2]{BellaicheChenevier-Book}
Bella{\"{\i}}che, J. and Chenevier, G.
\newblock {Families of {G}alois representations and {S}elmer groups}.
\newblock {\em Ast\'erisque} (2009), xii+314.

\bibitem[BE]{BreuilEmerton-Ordinary}
Breuil, C. and Emerton, M.
\newblock {Repr\'esentations {$p$}-adiques ordinaires de {${\rm GL}\sb 2(\bold
  Q\sb p)$} et compatibilit\'e local-global}.
\newblock {\em Ast\'erisque} (2010), 255--315.

\bibitem[BH]{BreuilHerzig-FundamentalAlgebraic}
Breuil, C. and Herzig, F.
\newblock {Ordinary representations of $G(\mathbf{Q}_p)$ and fundamental
  algebraic representations}.
\newblock {\em preprint} (2012).

\bibitem[Buz]{Buzzard-Eigenvarieties}
Buzzard, K.
\newblock {Eigenvarieties}.
\newblock In {\em {$L$}-functions and {G}alois representations}, volume 320 of
  {\em London Math. Soc. Lecture Note Ser.}, pages 59--120. Cambridge Univ.
  Press, Cambridge, 2007.

\bibitem[BT]{BuzzardTaylor-CompanionForms}
Buzzard, K. and Taylor, R.
\newblock {Companion forms and weight one forms}.
\newblock {\em Ann. of Math. (2)} {\textbf{149}}(1999), 905--919.

\bibitem[Che1]{Chenevier-padicAutomorphicForm}
Chenevier, G.
\newblock {Familles {$p$}-adiques de formes automorphes pour {${\rm GL}\sb
  n$}}.
\newblock {\em J. Reine Angew. Math.} {\textbf{570}}(2004), 143--217.

\bibitem[Che2]{Chenevier-InfiniteFern}
Chenevier, G.
\newblock {On the infinite fern of {G}alois representations of unitary type}.
\newblock {\em Ann. Sci. \'Ec. Norm. Sup\'er. (4)} {\textbf{44}}(2011),
  963--1019.

\bibitem[CM]{ColemanMazur-Eigencurve}
Coleman, R. and Mazur, B.
\newblock {The eigencurve}.
\newblock In {\em Galois representations in arithmetic algebraic geometry
  ({D}urham, 1996)}, volume 254 of {\em London Math. Soc. Lecture Note Ser.},
  pages 1--113. Cambridge Univ. Press, Cambridge, 1998.

\bibitem[Col1]{Coleman-padicShimura}
Coleman, R.~F.
\newblock {A {$p$}-adic {S}himura isomorphism and {$p$}-adic periods of modular
  forms}.
\newblock In {\em {$p$}-adic monodromy and the {B}irch and {S}winnerton-{D}yer
  conjecture ({B}oston, {MA}, 1991)}, volume 165 of {\em Contemp. Math.}, pages
  21--51. Amer. Math. Soc., Providence, RI, 1994.

\bibitem[Col2]{Coleman-ClassicalandOverconvergent}
Coleman, R.~F.
\newblock {Classical and overconvergent modular forms}.
\newblock {\em Invent. Math.} {\textbf{124}}(1996), 215--241.

\bibitem[Col3]{Coleman-OverconvergentModularFormsOfHigherLevel}
Coleman, R.~F.
\newblock {Classical and overconvergent modular forms of higher level}.
\newblock {\em J. Th\'eor. Nombres Bordeaux} {\textbf{9}}(1997), 395--403.

\bibitem[Gha]{Ghate-OrdinaryForms}
Ghate, E.
\newblock {Ordinary forms and their local {G}alois representations}.
\newblock In {\em Algebra and number theory}, pages 226--242. Hindustan Book
  Agency, Delhi, 2005.

\bibitem[Gro]{Gross-TamenessCriterion}
Gross, B.~H.
\newblock {A tameness criterion for {G}alois representations associated to
  modular forms (mod {$p$})}.
\newblock {\em Duke Math. J.} {\textbf{61}}(1990), 445--517.

\bibitem[Jon]{Jones-LocallyAnalytic}
Jones, O. T.~R.
\newblock {An analogue of the {BGG} resolution for locally analytic principal
  series}.
\newblock {\em J. Number Theory} {\textbf{131}}(2011), 1616--1640.

\bibitem[Kis]{Kisin-OverconvergentModularForms}
Kisin, M.
\newblock {Overconvergent modular forms and the {F}ontaine-{M}azur conjecture}.
\newblock {\em Invent. Math.} {\textbf{153}}(2003), 373--454.

\bibitem[Maz1]{Mazur-MSRI89-Deformations}
Mazur, B.
\newblock {Deforming {G}alois representations}.
\newblock In {\em Galois groups over {${\bf Q}$} ({B}erkeley, {CA}, 1987)},
  volume~16 of {\em Math. Sci. Res. Inst. Publ.}, pages 385--437. Springer, New
  York, 1989.

\bibitem[Maz2]{Mazur-padicvariation}
Mazur, B.
\newblock {The theme of {$p$}-adic variation}.
\newblock In {\em Mathematics: frontiers and perspectives}, pages 433--459.
  Amer. Math. Soc., Providence, RI, 2000.

\bibitem[Rou]{Rouqier-Jalgebra96-Pseudocharacters}
Rouquier, R.
\newblock {Caract\'erisation des caract\`eres et pseudo-caract\`eres}.
\newblock {\em J. Algebra} {\textbf{180}}(1996), 571--586.

\bibitem[Sch]{Scholl-MotivesForModularForms}
Scholl, A.~J.
\newblock {Motives for modular forms}.
\newblock {\em Invent. Math.} {\textbf{100}}(1990), 419--430.

\bibitem[Tay]{Taylor-Pseudoreps}
Taylor, R.
\newblock {Galois representations associated to {S}iegel modular forms of low
  weight}.
\newblock {\em Duke Math. J.} {\textbf{63}}(1991), 281--332.

\end{thebibliography}
\bibliographystyle{../../../bibliography/math}

\end{document}